\documentclass{amsproc}

\usepackage{amsmath}
\usepackage{amssymb}
\usepackage{graphicx}
\DeclareGraphicsExtensions{.png,.pdf,.eps}
\usepackage[all,2cell]{xy}
\usepackage{tikz}

\newtheorem{theorem}{Theorem}
\newtheorem{lemma}{Lemma}
\newtheorem{corollary}{Corollary}
\newtheorem{proposition}{Proposition}
\newtheorem{conjecture}{Conjecture}

\theoremstyle{definition}

\newtheorem{definition}{Definition}
\newtheorem{notation}{Notation}

\newtheorem{remark}{Remark}

\newtheorem{construction}{Construction}

\def\C{{\mathbb C}}

\def\N{{\mathbb N}}
\def\P{{\mathbb P}}
\def\Q{{\mathbb Q}}
\def\R{{\mathbb R}}
\def\Z{{\mathbb Z}}

\def\cD{{\mathcal D}}
\def\cE{{\mathcal E}}

\def\cO{{\mathcal{O}}}

\def\cV{{\mathcal V}}

\def\Q{{\mathbb{Q}}}

\def\fg{{\mathfrak g}}
\def\fh{{\mathfrak h}}
\def\fp{{\mathfrak p}}
\def\fb{{\mathfrak b}}

\def\sl{\operatorname{SL}}
\def\so{\operatorname{SO}}
\def\sp{\operatorname{Sp}}

\def\operatorname#1{\mathop{\rm #1}\nolimits}

\def\Pic{\operatorname{Pic}}

\def\det{\operatorname{det}}

\def\sga{{\langle}}
\def\sgc{{\rangle}}

\def\NE{{\operatorname{NE}}}

\def\Nef{{\operatorname{Nef}}}
\def\Nu{{\operatorname{N_1}}}

\def\Ch{{\operatorname{Ch}}}

\def\tr{\operatorname{tr}}
\def\Ad{\operatorname{Ad}}
\def\ad{\operatorname{ad}}
\def\End{\operatorname{End}}
\def\het{\operatorname{ht}}

\DeclareMathOperator{\ch}{\mathrm{Chains}}

\makeindex

\begin{document}
%\pagewiselinenumbers

\title[Rational curves, Dynkin diagrams and Fano manifolds]{Rational curves, Dynkin diagrams and Fano manifolds with nef tangent bundle}

%%    Information for first author
\author[R. Mu\~noz]{Roberto Mu\~noz}
%    Address of record for the research reported here
\address{Departamento de Matem\'atica Aplicada, ESCET, Universidad
Rey Juan Carlos, 28933-M\'os\-to\-les, Madrid, Spain}
\email{roberto.munoz@urjc.es}
\thanks{First and third author partially supported by the Spanish government project MTM2009-06964. Third author supported by National Researcher Program 2010-0020413 of NRF.
Third and fourth author partially supported by the Research in Pairs Program of CIRM.
Fourth author partially supported by JSPS KAKENHI Grant Number 24840008.}
%
%%    Information for second author
\author[G. Occhetta]{Gianluca Occhetta}
\address{Dipartimento di Matematica, Universit\`a di Trento, via
Sommarive 14 I-38123 Povo di Trento (TN), Italy}
\email{gianluca.occhetta@unitn.it}

%%    Information for third author
\author[L.E. Sol\'a Conde]{Luis E. Sol\'a Conde}
\address{%Departamento de Matem\'atica Aplicada, ESCET, Universidad
%Rey Juan Carlos, 28933-M\'ostoles, Madrid, Spain $\&$
Korea Institute for Advanced Study, 85 Hoegiro, Dongdaemun-gu, Seoul, 130--722, Korea}
\email{lesolac@gmail.com}

%%    General info

\author[K. Watanabe]{Kiwamu Watanabe}
\address{Course of Mathematics, Programs in Mathematics, Electronics and Informatics,
Graduate School of Science and Engineering, Saitama University.
Shimo-Okubo 255, Sakura-ku Saitama-shi, 338-8570 Japan}
\email{kwatanab@rimath.saitama-u.ac.jp}

\begin{abstract} A Fano manifold $X$ with nef tangent bundle is of Flag-Type if it has the same kind of elementary contractions as a complete flag manifold. In this paper we present a method to associate a Dynkin diagram $\mathcal{D}(X)$ with any such $X$, based on the numerical properties of its contractions. We then show that $\mathcal{D}(X)$ is the Dynkin diagram of a semisimple Lie group. As an application we prove that Campana-Peternell conjecture holds when $X$ is a Flag-Type manifold whose Dynkin diagram is $A_n$, i.e. we show that $X$ is the variety of complete flags of linear subspaces in $\mathbb{P}^n$.

\end{abstract}

\maketitle

%%%%%%%%%%%%%%%%%%%%
% FT-manifolds
% Section: Introduction
%%%%%%%%%%%%%%%%%%%%

\section{Introduction}\label{sec:intro}

The problem of classifying smooth complex projective varieties with nef tangent bundles appeared in the 1980's as the natural extension of the Hartshorne-Frankel Conjecture, proven by Mori using deformation theory of rational curves in his celebrated paper \cite{Mo}.

A quick look at the known examples shows that the natural candidates have homogeneous manifolds as building blocks. To be more precise, on one hand
Demailly, Peternell and Schneider showed in \cite{DPS} that, up to an \'etale covering, smooth varieties with nef tangent bundle always admit a fibration (their Albanese morphism) over an abelian variety, whose fibers are Fano manifolds with nef tangent bundle. On the other hand, Campana and Peternell have posed the following conjecture, which is the framework of the present paper:

\begin{conjecture}[Campana-Peternell Conjecture]\label{conj:CPconj}
The only complex Fano manifolds with nef tangent bundle are rational homogeneous spaces, i.e. quotients of semisimple Lie groups by parabolic subgroups.
\end{conjecture}

For brevity we will refer to Conjecture \ref{conj:CPconj} as {\it CP Conjecture}.
It has been proven to be true in several cases: for varieties of dimension three \cite{CP} and four \cite{CP2,Mk,Hw}, and  for fivefolds of Picard number bigger than one \cite{Wa2}. It holds also in other setups, for instance for smooth varieties with big and $1$-ample tangent bundle \cite{SW}.

Let us point out that in the quoted references the proofs depend on detailed classifications of the varieties satisfying the required properties, whose homogeneity is checked a posteriori.

In \cite[11.2]{CP} the authors suggested a possible approach to the problem, divided in two parts: first prove  CP Conjecture for smooth varieties with Picard number one; then prove that, given any Fano manifold with nef tangent bundle $X$ and a contraction $f:X \to Y$,
from the homogeneity of $Y$ and of the fibers of $f$ one can recover the homogeneity of $X$.
Unfortunately proving homogeneity in the Picard number one case turned out to be a very hard problem, and no progress has been made in this direction.

In this paper we propose to replace this ``bottom-up'' strategy with a ``top-down'' one, namely to start by proving  CP Conjecture for Fano manifolds of ``maximal'' Picard number and then proving that any Fano manifold with nef tangent bundle is dominated by one of such manifolds.

In order to give a reasonable notion of maximality
let us recall that given a semisimple Lie group $G$ and a parabolic subgroup $P$, taking $B$ to be a Borel subgroup containing $P$ we have a contraction $f:G/B \to G/P$, so the {\it complete flag manifold} $G/B$
dominates every $G$-homogeneous variety.

On the other direction, starting from a rational homogeneous manifold $G/P$, under the appropriate hypotheses (see Theorem \ref{thm:linesonflags}), one may define a contraction $f:G/P'\to G/P$ by considering the evaluation map of a family of minimal rational curves in $G/P$. The process stops when we get to a complete flag manifold, since $G/B$ is isomorphic to every universal family  of minimal curves contained in it.

In particular, complete flag manifolds associated with semisimple Lie groups can be recognized, among homogeneous manifolds, by the structure of their Mori contractions: all their elementary contractions are $\P^1$-bundles. This suggests considering {\it Flag-Type manifolds} (FT-manifolds for short), which we define as the Fano manifolds with nef tangent bundle whose elementary contractions are, more generally, smooth $\P^1$-fibrations.

The central idea in this paper is that for these manifolds we can face
directly the problem of constructing a semisimple group out of the nefness assumption. In fact,
from previous results of the authors (\cite{MOS,Wa}) one may prove that every FT-manifold of Picard number two is, in fact, homogeneous: this, apart of being the starting point of some inductive arguments, allows us to associate a Dynkin diagram $\cD(X)$ with every FT-manifold $X$. The Kac-Moody algebra encoded in this diagram could, in principle, be infinite dimensional, but
the main result of this paper is that this cannot happen, i.e. that we have the following:

\begin{theorem}\label{thm:picn}
The Dynkin diagram $\cD(X)$ of an FT-manifold $X$ is finite, that is, it determines a semisimple finite dimensional Lie algebra.
\end{theorem}

This result allows us to describe chains of rational curves on an FT-manifold $X$ in terms
of a monoid constructed using the Coxeter group associated with $\cD(X)$, using this result to prove
that the dimension of $X$ is bounded above by the dimension of the complete flag manifold
associated with $\cD(X)$.
In order to conclude  CP Conjecture for FT-manifolds, one has now to show that $X$ is, indeed, isomorphic to a quotient of the semisimple Lie group $G$  determined by $\cD(X)$; we show that,
without loss of generality, $G$ can be assumed to be simple (see Corollary \ref{cor:reducible}).

As en evidence in this direction we present, in the last section, a proof that this is the case for $\cD(X)=A_n$, obtaining that:

\begin{theorem}\label{thm:An}
Let $X$ be an FT-manifold with associated Dynkin diagram $A_n$. Then $X$ is isomorphic to the manifold parametrizing complete flags of linear subspaces in $\P^n$. In particular,  CP Conjecture holds for FT-manifolds whose associated Dynkin diagram is a disjoint union of diagrams of type $A$.
\end{theorem}

The strategy we use to prove Theorem \ref{thm:An} is the following: we show that the target of a particular contraction of $X$ is the expected one and then we obtain the result by an inductive process.
We also show that a similar result could be obtained for every other FT-manifold with finite Dynkin diagram,
provided that one can show that the target of a particular contraction is the expected one (see Proposition \ref{prop:recurs}). However we believe that a more general approach, which avoids a lot of case by case arguments, should be possible: this will be the goal of a forthcoming paper.

Let us finally note that, up to the complete classification of FT-manifolds,   CP Conjecture could be obtained from the solution of the following problem:
\begin{conjecture}\label{conj:CP2}
{\it Let $Y$ be a Fano manifold with nef tangent bundle which is not a product of positive dimensional varieties. Then there exists an FT-manifold $X$ with connected Dynkin diagram $\cD(X)$ dominating $Y$.}
\end{conjecture}
We discuss Conjecture \ref{conj:CP2} in Section \ref{ssec:init}, where we prove, as an evidence in this direction,  that the only Fano manifolds with nef tangent bundle dominating an FT-manifold $X$ are cartesian products with $X$ as a factor.

\noindent{\bf Acknowledgements:} This project was conceived when the third and fourth author enjoyed a grant of the ``Research in pairs'' program of the Fondazione Bruno Kessler, at the Centro Internazionale per la Ricerca Matematica (Trento). We would like to express our gratitude to this institution for its support and hospitality. We would also like to thank J. Wi\'sniewski for his interest in our project and his useful comments.

%%%%%%%%%%%%%%%%%%%%
% FT-manifolds
% Section: Ratl homog spaces
%%%%%%%%%%%%%%%%%%%%

\section{Preliminaries}\label{sec:nothom}

Throughout this paper all the algebraic varieties and morphisms will be defined over the complex numbers.

Since the general philosophy of this paper is that, in the appropriate setup, one may recover the rational homogeneous structure of a manifold out of certain families of rational curves contained in them, we will start by introducing some basic facts and notation regarding Mori theory and rational homogeneous spaces. We refer the interested reader to \cite{De} and \cite{H} for an account on each of this two matters.

\subsection{Contractions and rational curves on algebraic varieties}\label{ssec:mori}

Given any normal projective algebraic variety $X$ we will denote by $N_1(X)$ the vector space of $1$-cycles in $X$ with real coefficients, modulo numerical equivalence. Its dual vector space may be identified, via intersection theory, with the vector space $N^1(X)$ of real combinations of Cartier divisors in $X$ modulo numerical equivalence. The dimension of these vector spaces, that we denote by $\rho_X$, is usually called the {\it Picard number} of $X$. Note that, given a morphism $f:X\to Y$ between normal projective varieties, the pull-back of line bundles defines a linear map $f^*: N^1(Y)\to N^1(X)$, which is dual to the push-forward map of $1$-cycles.

Inside of $N_1(X)$ and $N^1(X)$ there are two dual convex cones that play a central role in Mori theory. On one hand the {\it Mori cone} of $X$ is defined as the closure $\overline{\NE}(X)$ of the convex cone generated by effective $1$-cycles. Its dual is, by  Kleiman's Ampleness Criterion, the closure of the cone generated by ample divisors, i.e. the cone of nef classes, or {\it nef cone} of $X$, $\Nef(X)$. Let us remark that, in the case of Fano manifolds, these two cones have a very special geometry:

\begin{theorem}[Cone Theorem for Fano manifolds]\label{lem:fanocone}
Let $X$ be a Fano manifold. Then $\overline{\NE}(X)=\NE(X)$ is rational polyhedral, i.e. there exist irreducible rational curves $C_1,\dots, C_m$ such that $\NE(X)$ is the closed convex cone generated by the classes $[C_i]$ of the curves $C_i$.
\end{theorem}

A surjective morphism with connected fibers $f: X \to Y$ between projective
normal varieties is called a {\it contraction} of $X$. A contraction $f$ is said to be of {\it fiber type} if it is not birational, that is if $\dim(X)>\dim(Y)$. Given a contraction $f$, the kernel of $f_*$ is a vector subspace of $N_1(X)$, meeting $\overline{\NE}(X)$ along a extremal face $\sigma$ that determines the contraction $f$. A contraction $f$ is called {\it elementary} if the corresponding face $\sigma$ is one dimensional, i.e. if $\sigma$ is an extremal ray. In the case of Fano manifolds, the relation between extremal rays of $\NE(X)$ and elementary contractions of $X$ is one to one, more concretely:

\begin{theorem}[Contraction Theorem for Fano manifolds]\label{lem:fanocont}
Let $X$ be a Fano manifold. Then for every extremal ray $R\subset\NE(X)$ there exists an elementary contraction $\varphi_R:X\to Y$ satisfying that, for every irreducible curve $C\subset X$, ${\varphi_R}_*([C])=0$ if and only if $\varphi_R(C)$ is a point.
\end{theorem}

\subsection{Semisimple Lie groups and algebras}\label{ssec:lie}

Along this section $G$ will denote a semisimple Lie group, with associated Lie algebra $\fg$.

\subsubsection{Cartan decomposition and root systems}
Given a {\it Cartan subalgebra} (that is, a maximal abelian subalgebra) $\fh\subset\fg$, its action on $\fg$ defines an eigenspace decomposition of $\fg$, called {\it Cartan decomposition} of $\fg$:
$$
\fg=\fh\oplus\bigoplus_{\alpha\in\fh^\vee\setminus\{0\}} \fg_\alpha.
$$
The spaces (so-called {\it root spaces}) $\fg_\alpha$ are defined by
$$\fg_\alpha=\left\{g\in\fg\,|\,[h,g]=\alpha(h)g,\mbox{ for all }h\in\fh\right\},$$
and the elements $\alpha\in\fh^\vee\setminus\{0\}$ for which $\fg_\alpha\neq 0$ are called {\it roots} of $\fg$; the finite set of these elements is called the {\it root system} of $\fg$ and will be denoted by $\Phi$. Moreover, one may prove that the root spaces are one-dimensional, that $\alpha\in\Phi$ iff $-\alpha\in\Phi$, and that $[\fg_\alpha,\fg_\beta]=\fg_{\alpha+\beta}$ iff $0\neq\alpha+\beta\in\Phi$.

\subsubsection{The Weyl group of $\fg$}
The {\it Killing form} $\kappa(X,Y):=\tr(\ad_X\circ\ad_Y)$ defines a nondegenerate
bilinear form on $\fh$, whose restriction to the real vector space $E$ generated by $\Phi$ is positive definite.
It is then well known that the root system $\Phi$ is invariant by the {\it reflections with respect to roots}, which are the isometries of $(E,\kappa)$ defined by:
$$
\sigma_\alpha(x)=x-\sga x,\alpha\sgc\alpha,\quad\mbox{where}\quad\sga x,\alpha\sgc:=2\dfrac{\kappa(x,\alpha)}{\kappa(\alpha,\alpha)}.
$$
The group $W\subset\so(E,\kappa)$ generated by the $\sigma_\alpha$'s is called the {\it Weyl group of $\fg$}.

\subsubsection{The Cartan matrix of $\fg$}\label{ssec:cartan}

Set $n:=\dim_\C(\fh)$ and $D:=\{1,2,\dots,n\}$.
A basis of $\fh^\vee$ formed by elements of $\Phi$ such that the coordinates of every element of $\Phi$ are integers, all of them nonnegative or all of them nonpositive, is called a {\it base} of $\Phi$ and its elements are called {\it simple roots}. It is known that such a set always exists, and we will choose one of them, denoting it by $\Delta=\left\{\alpha_i\right\}_{i\in D}$. It provides a decomposition of the set of roots according to their sign $\Phi=\Phi^+\cup\Phi^-$, where $\Phi^-=-\Phi^+$.  Moreover, every positive root can be obtained from simple roots by means of reflections $\sigma_{\alpha_i}$. It is then clear that the matrix $M$ whose coefficients are $\sga\alpha_i,\alpha_j\sgc$, the so-called  {\it Cartan matrix of $\fg$}, encodes the necessary information to reconstruct $\fg$ from the set of simple roots $\Delta$.

\subsubsection{Dynkin diagrams and the classification of semisimple Lie groups}\label{ssec:dynkinhomo}

The coefficients of the Cartan matrix $M$ of $\fg$ are subject to certain arithmetic restrictions:
\begin{itemize}
\item $\sga \alpha_i,\alpha_i\sgc=2$ for all $i$,
\item $\sga \alpha_i,\alpha_j\sgc=0$ if and only if $\sga \alpha_j,\alpha_i\sgc=0$, and
\item if $\sga \alpha_i,\alpha_j\sgc\neq 0$, $i \neq j$, then $\sga \alpha_i,\alpha_j\sgc\in\Z^-$ and $ \sga \alpha_i,\alpha_j\sgc\sga\alpha_j,\alpha_i\sgc=1,2$ or $3$.
\end{itemize}
These properties allow us to represent $M$ by a {\it Dynkin diagram}, that we denote by $\cD$, that consists of a graph whose set of nodes is $D$ and where the nodes $i$ and $j$ are joined by $\sga \alpha_j,\alpha_i\sgc\sga \alpha_i,\alpha_j\sgc$ edges. When two nodes $i$ and $j$ are joined by a double or triple edge, we add to it an arrow, pointing to $i$ if $\sga \alpha_i,\alpha_j\sgc>\sga \alpha_j,\alpha_i\sgc$.

The classification theorem of root systems tells us that every finite root system is a disjoint union of mutually orthogonal irreducible root subsystems, each of them corresponding to one of the {\it connected finite Dynkin diagrams} $A_n$, $B_n$, $C_n$, $D_n$ ($n\in\N$), $E_6$, $E_7$, $E_8$, $F_4$, $G_2$.

From this result it follows that every semisimple Lie group $G$ is a product of {\it simple} Lie groups, each of them corresponding to one of the Dynkin diagrams above. Recall that under this correspondence, the classical Lie groups $\sl_{n+1}$, $\so_{2n+1}$, $\sp_{2n}$ and $\so_{2n}$ correspond to the diagrams $A_n$, $B_n$, $C_n$ and $D_n$, respectively.

\subsection{Rational homogeneous spaces}\label{ssec:rathom}

A rational homogeneous space is a projective smooth manifold isomorphic to a quotient $G/P$ where $G$ is a semisimple group and $P$ is a subgroup of $G$. Let us start by recalling how these quotients may be described in terms of root systems.

\subsubsection{Marked Dynkin diagrams}

The subgroups of $G$ for which $G/P$ is a projective manifold are called {\it parabolic}. The most important thing to note here is that every parabolic subgroup is determined by a set of simple roots of $G$ in the following way: given a subset $I\subset D$, let $\Phi^+(I)$ be the subset of $\Phi^+$ generated by the simple roots in $D\setminus I$; then the subspace
\begin{equation}
\fp(I):=\fh\oplus\bigoplus_{\alpha\in\Phi^+} \fg_{-\alpha}\oplus\bigoplus_{\alpha\in\Phi^+(I)} \fg_\alpha
\label{eq:cartanparab}
\end{equation}
is a {\it parabolic subalgebra} of $\fg$, determining a parabolic subgroup  $P(I)\subset G$. Conversely, every parabolic subgroup is constructed in this way. In the most common notation, we represent $G/P(I)$ by marking on the Dynkin diagram $\cD$ of $G$ the nodes corresponding to $I$. When $G$ is clear from the context, we will denote by $F(I)$ the projective manifold $G/P(I)$.

%%%%%%%%%%%%%%%%%%%%%%%%%%%%%%%%%%%%%
\subsubsection{Contractions of rational homogeneous spaces}\label{ssection:contraction}

From the above construction  it immediately follows that given two subsets $I\subset J\subset D$, there exists a  proper surjective morphism $p^{J,I}:F(J)\to F(I)$.
We will denote by $T^{J,I}$ the relative tangent bundle of $p^{J,I}$, and by $K^{J,I}$ its relative canonical divisor. The following result describes the Mori cone of $F(I)$:

\begin{proposition}\label{prop:simp}
Every rational homogeneous manifold $F(I)=G/P(I)$ is a Fano manifold, whose contractions are all of the form $p^{J,I}$, $I\subset J\subset D$. In particular, the Picard number of $F(I)$ is $\sharp(I)$ and $\overline{\NE}(F(I))\subset N_1(F(I))$ is simplicial. Moreover, the fibers of a contraction $p^{J,I}$ are rational homogeneous spaces, determined by the marked Dynkin diagram obtained from $\cD$ by removing the nodes in $I$ and marking $J\setminus I$.

\end{proposition}

\subsubsection{Complete flag manifolds}\label{ssection:flagmanifolds}

The smallest parabolic subgroup (up to the choice of a Cartan subgroup and a set of simple roots)  $B:=P(D)$ receives the name of {\it Borel subgroup} of $G$, and the corresponding rational homogeneous space $F:=F(D)=G/B$ is usually called the {\it complete flag manifold} of $G$. For simplicity, given any subset $I\subset D$, we will write $p^I,T^I, K^I$ instead of  $p^{D,I},T^{D,I}, K^{D,I}$. Moreover, for $I=\{i\}$ we will use the index $i$ instead of $\{i\}$. Finally, for every subset $I\subset D$, we will use the notation $p_I, T_I, K_I$ to denote $p^{D\setminus I},T^{D\setminus I}, K^{D\setminus I}$. For instance: the relative canonical divisors of the $i$-th elementary contraction of $F$ will be denoted by $K_i$. Note that, under this description, one may prove the following:
\begin{proposition}\label{prop:flagfibers}
The fibers of every contraction of $F=G/B$ are complete flag manifolds.
\end{proposition}

\subsection{Homogeneous vector bundles on rational homogeneous spaces}\label{ssec:bundle}

In this section we will recall some basic results on
homogeneous vector bundles
that we will need later on.

\subsubsection{Relative tangent bundles of contractions of a complete flag manifold}
Homogeneous vector bundles on a rational homogeneous manifold $G/P$ are determined by representations of the Lie algebra $\fp$. For instance, the {\it tangent bundle} of $G/P$, whose total space may be described as
$$
T_{G/P}=G\times_P\fg/\fp:=(G\times\fg/\fp)/\sim,
$$
where  $(g,v+\fp)\sim(gp^{-1},\Ad_{p}(v)+\fp)$ for all $p\in P$,
is given by the representation:
$$
\fp\to\End(\fg/\fp)\mbox{ given by }X\mapsto \ad_X\,\,(\mbox{mod }\fp).
$$

In the case of a complete flag manifold $F=G/B$, we may identify $\fg/\fb$ with ${\bigoplus_{\beta\in\Phi^+}\fg_\beta}$, and then one sees that for every subset $I\subset D$, the subspace ${\bigoplus_{\beta\in\Phi^+(I)}\fg_\beta}$ is invariant by the action of $\fb$. The homogeneous vector bundle that this $\fb$-module defines is, in fact, the relative tangent bundle $T^I$ of the contraction $p^I:F\to F(I)$. In particular, for every index $i\in D$ the relative tangent bundle $T_i$ is isomorphic to the line subbundle:
$$
G\times_B\fg_{\alpha_i}\hookrightarrow T_{F}.
$$

\subsubsection{Line bundles, weights and $1$-cycles}

It is a known fact that every line bundle on a rational homogeneous space is homogeneous. As a consequence, the Picard group of a manifold $F(I)$ can be written in terms of the marked diagram determining it.

In fact, every homogeneous line bundle on $F(I)$ corresponds to a representation of $P(I)$ over $\C$ or, equivalently, to a morphism of Lie algebras $\lambda:\fp(I)\to\C$. Using Equation (\ref{eq:cartanparab}), one may easily see that $\lambda$ factors via the projection to $\fh$ and that the corresponding element in $\fh^\vee$ satisfies $\sga \lambda,\alpha \sgc\in \Z$ for all $\alpha\in\Phi$ and $\sga \lambda,\alpha \sgc=0$ for all $\alpha\in\Phi^+(I)$.

In particular the set of {\it weights} of $G$
$$
\Lambda:=\left\{\lambda\in\fh^\vee|\,\sga \lambda,\alpha \sgc\in \Z\mbox{ for all }\alpha\in\Phi\right\}
$$
may be identified with $\Pic(F)$ and, under this identification, the Picard group of any other $F(I)$ may be seen, via the corresponding pull-back map to $F$, as the sublattice of $\Lambda$ orthogonal to $\alpha_i$ for all $i\in D\setminus I$.
For instance, under this identification the weights $-\alpha_i$ correspond to the relative tangent bundles $T_i$  of the contractions $p_i:F\to F(D\setminus\{i\})$.

Moreover, this description of line bundles as weights allows us to identify the linear operators of the form $\sga\,\,,\alpha\sgc$ with numerical classes of $1$-cycles in $G/B$. For instance, for any $i\in D$, the operator $\sga\,\,,\alpha_i\sgc$ vanishes, by construction, on $p_i^*\Pic(F(D\setminus\{i\}))$, and satisfies $\sga\alpha_i,\alpha_i\sgc=2$, hence it corresponds to the class of the fiber $C_i$ of the elementary contraction $p_i:F\to F(D\setminus\{i\})$. In particular:
\begin{proposition}\label{prop:cartanint}
With the same notation as above, the Cartan matrix of $G$ is equal to the intersection matrix $(T_i\cdot C_j)$.
\end{proposition}

As a final consequence, the dual base $\left\{\lambda_j\right\}$ of the base of simple roots $\left\{\alpha_i\right\}$ (the so-called set of {\it fundamental weights} of $G$) may be then identified with the set of pull-backs of the ample generators of the Picard groups of the manifolds $F(i)$.

\subsubsection{The anticanonical bundle of $F$.}

From the homogeneous description of the tangent bundle one may show that the anticanonical divisor $-K_{F}$ corresponds to the weight $\sum_{\alpha\in\Phi^+} \alpha$. Similarly, it follows that the relative anticanonical divisor $-K^I$ of the contraction $p^I$ corresponds to the weight $\sum_{\alpha\in\Phi^+(I)} \alpha$. In particular, it is an integral combination of the relative anticanonical divisors $-K_i$. In the case of $-K_{F}$, for instance, the coefficients may be easily computed from the Cartan matrix of $G$, by using Prop. \ref{prop:cartanint} and taking in account that $-K_{F}\cdot C_i=2$ for all $i$.

\subsubsection{Homogeneous filtrations of the tangent bundle}\label{ssec:tangentbundle}

Every homogeneous vector subbundle of $T_F$ is determined by a subspace
$\bigoplus_{\beta\in\Psi}\fg_\beta\hookrightarrow\fg/\fb$ that is $\fb$-invariant or, equivalently, by a subset $\Psi\subset\Phi^+$ satisfying the following property:
\begin{equation}
\alpha+\beta\in\Psi, \alpha,\beta\in\Phi^+\mbox{ implies that } \alpha,\beta\in\Psi.
\end{equation}
We then say that $\Psi$ is {\it admissible}.

Note that one may find a finite sequence of admissible subsets,
$$\Psi_1\subsetneq\dots\subsetneq\Psi_m=\Phi^+$$ satisfying that $\Psi_{i}=\Psi_{i-1}\cup\{\beta_i\}$, for some $\beta_i\in\Phi^+$, for all $i$. In fact, denoting by $\het(\beta)$ the height of a root $\beta$ with respect to the base $\Delta$, that is $\het(\beta)=\sum_{k=1}^nr_{j}$ for $\beta=\sum_{k=1}^nr_{j}\alpha_j$, consider any ordering of $\Phi^+$,
$\left\{\beta_{i_1},\cdots,\beta_{i_m}\right\}$ satisfying that $\het(\beta_j)\leq\het(\beta_{j+1})$ for every $j$. Then the sets $\Psi_k=\left\{\beta_{i_1},\cdots,\beta_{i_k}\right\}$ are admissible.

In particular we obtain a filtration of
$T_{F}$:
$$
T_{\Psi_1}\subsetneq\dots\subsetneq T_{\Psi_m}=T_{F}
$$
satisfying that:
$$T_{\Psi_i}/T_{\Psi_{i-1}}\cong G\times_B\fg_{\beta_i}.$$
Note that, if $\beta_i=\sum_j r_j\alpha_j$, then
$$
 G\times_B\fg_{\beta_i}=\bigotimes_j\left(G\times_B\fg_{\alpha_j}\right)^{\otimes r_j}.
$$

\begin{construction}\label{const:foliation}
The tangent bundle $T_F$ may be described as the involutive closure of the direct sum of the line bundles $T_\ell$, $\ell\in D$. We will show here a stepwise procedure to construct a homogeneous filtration of $T_F$ starting from the line bundles $T_\ell$, via Lie brackets. In order to see this, let us denote
$$V_{k}:=\bigoplus_{\beta\in\Psi_k}\fg_\beta$$
the $k$-th element of the filtration of $\fg/\fb$ by $\fb$-submodules, whose subsequent quotients $V_k/V_{k-1}$ are isomorphic to $\fg_{\beta_k}$. We will impose the ordering of $\Phi^+$ to satisfy that the elements of the same height appear in lexicographic order with respect to their coordinates in the base $\Delta$.

It then follows that, for every $k\in\{n+1,\dots,m\}$ there exist $j<k$ and $\ell\in\{1,\dots,n\}$ satisfying that:
\begin{itemize}
\item $\beta_k=\beta_j+\alpha_\ell$, and
\item $\beta_{j'}+\alpha_\ell\in V_k$ for every $j'<j$.
\end{itemize}

Then, for every $k$, the Lie bracket provides a morphism of $B$-modules:
\begin{equation}\label{eq:liebr}
\nu_k:V_j\otimes\fg_{\alpha_\ell}\to \dfrac{V_m}{V_{k-1}}=\dfrac{\fg/\fb}{V_{k-1}},
\end{equation}
whose restriction to $V_{j-1}\otimes\fg_{\alpha_\ell}$ is zero. In particular we get a commutative diagram of $B$-modules:
$$
\xymatrix{V_j\otimes\fg_{\alpha_\ell}\ar[r]^{\nu_k}\ar[d]&V_m/V_{k-1}&V_m\ar[l]\\
\fg_{\beta_j}\otimes\fg_{\alpha_\ell}\ar[r]^{\sim}&\fg_{\beta_k}\ar@{^{(}->}[u]&V_{k}\ar@{^{(}->}[u]\ar[l]}
$$
We may now translate this into the language of homogeneous vector bundles over $F$. Note first that the Lie bracket morphism (\ref{eq:liebr}) defines an {\it O'Neill tensor} of distributions in $T_{F}$, that is the $\cO_{F}$-linear morphism given by the composition of the usual Lie bracket with the quotient modulo $T_{\Psi_{k-1}}$:
$$
N_k:T_{\Psi_j}\otimes T_\ell \longrightarrow T_{F}/T_{\Psi_{k-1}},
$$
at step $k$, there exists $j<k$ such that the vector subbundle $T_{\Psi_{k}}$ may be defined as the inverse image in $T_{F}$ of the image of $N_k$ (which is isomorphic to $G\times_B \fg_{\beta_k}$).
\end{construction}

%%%%%%%%%%%%%%%%%%%%
% FT-manifolds
% Section: Basics on varieties with TX nef
%%%%%%%%%%%%%%%%%%%%

\section{Flag-Type manifolds}\label{sec:basics}

In this section we will introduce the definition of Flag-Type manifolds and present their basic features.
We begin by stating a set of well-known properties of varieties with nef tangent bundle, paying special attention to their Mori cones.

\begin{proposition}\label{prop:nef}
Let $X$ be a smooth complex Fano manifold with nef tangent bundle. Then the following properties hold:
\begin{enumerate}
\item Every contraction $\pi:X\to Y$ is of fiber type, i.e. $\dim(Y)<\dim(X)$.
\item Every contraction $\pi:X\to Y$ is smooth and, moreover, its image $Y$ and every fiber $\pi^ {-1}(y)$ are Fano manifolds with nef tangent bundle.
\item For every contraction $\pi:X\to Y$, the Picard number of a fiber $\pi^{-1}(y)$ equals $\rho_X-\rho_Y$. Moreover, being $j:\pi^{-1}(y) \to X$ the inclusion and $j_*:\Nu(\pi^{-1}(y)) \to \Nu(X)$ the induced linear map, we have  $j_*(\NE(\pi^{-1}(y)))= \NE(X)\cap j_*(\Nu(\pi^{-1}(y)))$.
\item The Mori cone $\NE(X)$ is simplicial.
\end{enumerate}
\end{proposition}

\begin{proof}
The first part follows from the fact that the fibers of every contraction of $X$  contain rational curves. Since $T_X$ is nef, every rational curve is free, and the claim follows. For the second part we refer the reader to \cite[Thm.~5.2]{DPS} and \cite[Thm.~4.4]{SW}.

The equality $\rho_{\pi^{-1}(y)}=\rho_X-\rho_Y$ follows from \cite[Example 3.8]{Ca}; as a consequence we have
 that $j_*(\Nu(\pi^{-1}(y)))= \ker \pi_*$.

Denote by $\sigma=\NE(X) \cap \ker \pi_*$ the face determining the contraction $\pi$. For every extremal ray  $R$ contained in $\sigma$, since the associated contraction is of fiber type, there is a curve $C$ such that $[C] \in R$ and  $C\cap \pi^{-1}(y) \not = \emptyset$. Since $C$ is contracted to a point, necessarily  $C \subset \pi^{-1}(y)$.
We can thus conclude that  $j_*(\NE(\pi^{-1}(y)))= \sigma$.

To prove the last statement let us assume by contradiction that $\NE(X)$ is not simplicial.
Let $R_1, \dots,R_m$, $m >n:=\rho_X$ be the extremal rays of $\NE(X)$, ordered in such a way that
$\sga R_1, \dots, R_n\sgc =\Nu(X)$. For every $i=1, \dots, m$ denote by $\Gamma_i$ a rational curve of minimal anticanonical degree
among those such that $[\Gamma_i] \in R_i$. We can write $[\Gamma_m]= \sum_{i=1}^n a_i [\Gamma_i]$, with $a_i \in \Q$. By the extremality of $R_m$ there exists $j$
such that $a_j <0$; without loss of generality, we may assume that $j=1$.

For every $i \in  \{2, \dots, n\}$ let $V_i$ be the family of rational curves containing $\Gamma_i$,
which is unsplit and covering by the minimality of $-K_X \cdot \Gamma_i$ and the nefness of $T_X$.
We apply to these families \cite[Lemma 2.4]{CO1}: since every contraction of $X$ is of fiber type, the classes $[V_2], \dots, [V_n]$ must lie in an $(n-1)$-dimensional extremal face of $\NE(X)$. If $H$ is a supporting divisor of this face, then $H \cdot \Gamma_i=0$ for $i \in  \{2, \dots, n\}$ and $H \cdot \Gamma_1 >0$, so that $H \cdot \Gamma_m <0$. This contradicts the nefness of $H$.
\end{proof}

\begin{definition}\label{defn:FT}
We say that a Fano manifold $X$  with nef tangent bundle is a {\it Flag-Type manifold} or, for brevity, an {\it FT-manifold}, if every elementary contraction of $X$ is a smooth $\P^1$-fibration.
\end{definition}

\begin{remark}\label{rem:P1bdl}
The expression {\it $\P^1$-bundle} appears in the literature with different meanings. Within this paper we will use it to refer to the Grothendieck projectivization of a rank two vector bundle, whereas {\it smooth $\P^1$-fibration} will refer to a smooth morphism with fibers isomorphic to $\P^1$. Note that the two concepts coincide if the base variety has trivial Brauer group (over a curve, for instance).

Complete flag manifolds are FT-manifolds but, moreover, their contractions are $\P^1$-bundles. In fact, by Proposition \ref{prop:flagfibers}, given a complete flag manifold $F=G/B$ and an index $i\in D$, the fibers of the elementary contraction $p_i:F\to F(D \setminus \{i\})$ are $\P^1$'s and it is enough to note that the line bundle $L_i$ associated with the fundamental weight $\lambda_i$ has intersection number one with the fibers of $p_i$, so that $F$ is isomorphic to the Grothendieck projectivization of $\P({p_i}_*\cO(L_i))$.
\end{remark}

%%%%%%%%%%%%%%%%%%%%%%%%%%%%%%%%%%%%%%
\subsection{Initiality of FT-manifolds}\label{ssec:init}

In this section we will make a couple of remarks regarding Conjecture \ref{conj:CP2}. We will discuss first how this result would imply  CP Conjecture via the homogeneity of FT-manifolds.

\begin{remark}\label{rem:CP3} Let $Y$ be a Fano manifold with nef tangent bundle and assume, without loss of generality, that $Y$ is not a product of two positive dimensional manifolds. If Conjecture \ref{conj:CP2} holds, then there exists a surjective morphism $f:X\to Y$ from an FT-manifold $X$ with connected Dynkin diagram $\cD(X)$. Then the homogeneity of $X$  would imply the homogeneity of the image $Y'$ of the Stein factorization of $f$. Since $\cD(X)$ is connected, $Y'$ would be the quotient of a simple Lie group so, by  \cite[Main Theorem]{Lau}, either the corresponding finite morphism $\mu:Y'\to Y$ is an isomorphism, or $Y$ is a projective space. In any case $Y$ would be homogeneous.
\end{remark}

We will refer to Conjecture \ref{conj:CP2} as the problem of {\it initiality} of FT-manifolds. An evidence for this conjecture is the fact that the only Fano manifolds with nef tangent bundle dominating an FT-manifold $X$ are constructed in the trivial way, namely as cartesian products of $X$. We will show this in Proposition \ref{prop:maxevid} below, for which we need to prove first that the statement holds in the case $X=\P^1$:

\begin{lemma}\label{lem:maxevid}
Let $M$ be a Fano manifold with nef tangent bundle admitting a contraction   $f:M \to \P^1$. Then there exists a smooth variety $Z$ such that $M\cong \P^1\times Z$.
\end{lemma}

\begin{proof}
Denote by $\sigma$ the face of $\NE(M)$ generated by the rays that are contracted by $f$ and by $R$ the (unique) extremal ray of $\NE(M)$ not contracted by $f$. By Proposition \ref{prop:nef}, the contraction $g:M\to Z$ associated to $R$ is smooth and, moreover, its fibers are isomorphic to $\P^1$: in fact they are Fano manifolds and the restriction of $f$ to each of them is a finite morphism.

Since the fibers of $f$ are rationally connected, by \cite[Theorem 1.1]{GHS} $f$ admits a section. Let $C \subset M$ be a section whose intersection number with $-K_M$ is minimal and let $V$ be a family of rational curves containing $C$; this family is covering since $T_M$ is nef, and it is unsplit by the minimality of $C$.
We claim that $[C] \in R$. Assuming by contradiction that this is not the case,
let $C'$ be the normalization of $g(C)$, and $S'$ be the normalization of $S:=g^{-1}(g(C))$: $S'$ is a smooth $\P^1$-fibration over $C'$, hence a $\P^1$-bundle. Note that the embedding $C\subset S$ lifts to an inclusion $C\to S'$.  Since the restriction of $f$ to $S$ provides a second fiber type contraction of this surface, it follows that $S' \cong\P^1\times\P^1$. But in $\P^1\times\P^1$ there exists a degeneration of $C$ to a cycle consisting of lines of the two rulings, contradicting the minimality of $C$.

We have thus proved that sections of $f$ of minimal degree are contracted by $g$; since fibers of $g$ are numerically equivalent we have that the fibers of $g$ are sections of $f$. In particular the natural morphism $(f,g):M\to\P^1\times Z$ is injective, hence an isomorphism.
\end{proof}

\begin{proposition}\label{prop:maxevid}
Let $M$ be a Fano manifold with nef tangent bundle which admits a contraction $f:M \to X$ onto an FT-manifold $X$. Then there exists a smooth variety $Y$ such that $M\cong X\times Y$.
\end{proposition}

\begin{proof}

Denote by $\sigma$ the face of $\NE(M)$ generated by the rays that are contracted by $f$,  by $\overline{\sigma}$
the face spanned by the other extremal rays and by $g:M \to Y$
the contraction determined by the face $\overline{\sigma}$. We will show that, for every $y\in Y$, the restriction $f_{|g^{-1}(y)}:g^{-1}(y)\to X$ is an isomorphism, from which one easily concludes that $(f,g):M\to X\times Y$ is an isomorphism, too.

Denote by $R_i$, $i=1,\dots,n$, the extremal rays of $\overline{\sigma}$ and by $\tau_i$ the face generated by $\sigma$ and $R_i$. The contraction of $\tau_i$ factors through $f$  and an elementary contraction  $\pi_i:X\to X_i$.  Moreover, since $\rho_X= \rho_M - \sharp(\sigma) = \sharp(\overline{\sigma})$, every elementary contraction of $X$
can be given in this way. For every $i$, the inverse image $M_i:=f^{-1}(\Gamma_i)$ of a fiber $\Gamma_i$ of $\pi_i$ is, by Proposition \ref{prop:nef}, a Fano manifold with nef tangent bundle. Moreover, since $M_i$ admits a contraction to $\Gamma_i\cong\P^1$, by Lemma \ref{lem:maxevid}, we obtain that $M_i\cong\Gamma_i\times Z$, for some smooth variety $Z$.

Since $X$ is chain connected by curves of the form $\Gamma_1,\dots,\Gamma_n$, the variety $Z$ does not depend on the choice of the curve $\Gamma_i$, nor on the choice of the index $i$. Note also that, by Proposition \ref{prop:nef}(3), the Stein factorization of $g_{|M_i}$ is the canonical projection from $M_i$ to $Z$. Let us denote by $\Gamma_i'$ the fibers of this projection, and note that they are sections of $f$ over $\Gamma_i$ contracted by $g$.

Consider a point $z\in g^{-1}(y)$, $x=f(z)$, and let $x'$ be any point in $X$. Given a chain of rational curves
$(\Gamma_{i_1},\dots,\Gamma_{i_m})$ in $X$ joining $x$ with $x'$, we may choose sections over the curves of the chain so that
$z\in \Gamma'_{i_1}$ and $\Gamma'_{i_k}\cap \Gamma'_{i_{k+1}}\neq\emptyset$ for all $k$. In particular the whole chain is contained
in $g^{-1}(y)$ and it contains a point $z'$ satisfying $f(z')=x'$. It follows then that $f_{|g^{-1}(y)}$ is surjective.

Now, since $X$ is Fano, the proof may be concluded by showing that the numerical class of the ramification divisor of $f_{|g^{-1}(y)}$, that we denote by $B\subset g^{-1}(y)$, is zero. By construction, the classes of $\Gamma_i'$'s generate $N_1(g^{-1}(y))$, so the proof will be finished by showing that $B\cdot \Gamma_i'$ is zero for every $i$.

The intersection $g^{-1}(y)\cap M_i$ is  a union of sections of the form $\Gamma_i'$, therefore, for a general $x \in X$, the morphism $f_{|g^{-1}(y)}$ is not ramified over the inverse image of the curves $\Gamma_i$ passing through $x$.
%if  and a point $x$ belongs to the ramification divisor if and only if the section $\Gamma_i'$ passing by $x$ lies completely in it.
So, if $\bar x  \in g^{-1}(y)$ is general, the sections $\Gamma_i'$ passing through $\bar x$ do not meet the ramification divisor $B$, hence  $B \cdot \Gamma_i'=0$ for every $i$.
\end{proof}

%%%%%%%%%%%%%%%%%%%%%%%%%%%%%%%%%%%%%%

\subsection{The Dynkin diagram of an FT-manifold} \label{ssec:dynkin}

By analogy with the notation for complete flag manifolds, introduced  in \ref{ssection:contraction} and \ref{ssection:flagmanifolds}, along the rest of the paper we will use the following notation for FT-manifolds.

\begin{notation}\label{not:ftmanifold}
Let $X$ be an FT-manifold of Picard number $n$; we will denote by $R_i$,  $i=1, \dots, n$ its extremal rays, and by $\Gamma_i$ a rational curve of minimal degree such that $[\Gamma_i] \in R_i$.
If $I$ is any subset of $D:=\{1, \dots, n\}$ we will denote by $R_I$ the extremal face spanned by the rays $R_i$ such that $i \in I$, by $\pi_I:X \to X_I$ the corresponding extremal contraction, by $T_I$ the relative tangent bundle and by $K_I:= -\det T_I$. We will also denote by $\pi^I:X\to X^I$ the contraction of the face $R^I$ spanned by the rays $R_i$ such that $i \in D \setminus I$.
For $I \subset J \subset D $ we will denote the
contraction of the extremal face $\pi_{I*}(R_J) \subset \Nu(X_I)$ by $\pi_{I,J}:X_I \to X_J$ or by
$\pi^{D \setminus I,D \setminus J}:X^{D \setminus I} \to X^{D\setminus J}$.
\end{notation}

The geometric interpretation of the Cartan matrix of a Lie algebra $\fg$ in terms of the intersection theory of $G/B$ given in Proposition \ref{prop:cartanint} suggests a natural way of associating a Cartan matrix with every FT-manifold:

\begin{definition}
Let $X$ be an FT-manifold of Picard number $n$. The {\it Cartan matrix} $M(X)$ associated with $X$ is the $n \times n$ matrix
defined by $M(X)_{ij}=-K_i \cdot \Gamma_j$.
\end{definition}

We will denote by $M_I(X)$ the $|I| \times |I|$ submatrix of $M(X)$ obtained from $M(X)$ by subtracting rows and columns corresponding to indices which are not in $I$.
The next  proposition shows that the fibers of a contraction
of an FT-manifold $X$  are FT-manifolds, whose Cartan matrices
are submatrices of
$M(X)$.

\begin{proposition}\label{prop:fibers} Let $X$ be an $FT$-manifold and consider any subset $I\subset D$. Let $\pi_I:X \to X_I$ be the contraction of a face $R_I$ and denote by $Z_I$  a fiber of $\pi_I$. Then $Z_I$ is an FT-manifold
whose Cartan matrix of $Z_I$ is $M_I(X)$.
\end{proposition}

\begin{proof} By  Proposition \ref{prop:nef} we already know that $Z_I$ is a Fano manifold with nef tangent bundle, such that $\rho_{Z_I}=\sharp I=\rho_X -\rho_{X_I}$, whose nef cone may be naturally identified with $R_I$.

Let  $\pi_i:X \to X_i$ be an elementary contraction such that $i \in I$. Its restriction $(\pi_i)|_{Z_I}$ is an extremal contraction of $Z_I$, which is smooth and of fiber type in view of Proposition \ref{prop:nef}, and whose fibers are obviously $\P^1$'s. Hence it is a smooth $\P^1$-fibration.
The last assertion follows from the fact that relative tangent bundles corresponding to elementary contractions $\pi_i$ indexed by  $i \in I$ restrict to the relative tangent bundles of the restriction $(\pi_i)|_{Z_I}$.
\end{proof}

We finish this section by stating one of the key ingredients of this paper: the fact that CP Conjecture holds for FT-manifolds of Picard number two.

\begin{theorem}\label{thm:pic2}
Let $X$ be an FT-manifold of Picard number two. Then $X$ is isomorphic to $G/B$ with $G$  a semisimple Lie group of type $A_1\times A_1$, $A_2$, $B_2$ or $G_2$.
\end{theorem}

%%%%%%%%%%%%%%%%%%%%%%%%%%%%%%%%%%%

Let $X$ be an $(m+1)$-dimensional Fano manifold of Picard number two having two $\P^1$-fibrations $\pi_1:X\to X_1$, $\pi_2:X\to X_2$.
Let $H_i$ be the pull-back of the ample generator of $X_i$, $i=1,2$ and set $\nu_1:=K_1\cdot \Gamma_2$, $\mu_1:=H_1\cdot \Gamma_2$, $\nu_2:=K_2\cdot \Gamma_1$, $\mu_2:=H_2\cdot \Gamma_1$, so that one may easily verify that:
\begin{equation}\label{eq:basechange}
\begin{pmatrix}-K_1\\H_1\end{pmatrix}
 =A\begin{pmatrix}-K_2\\H_2\end{pmatrix},\mbox{ where~}
A:=\begin{pmatrix}\vspace{0.2cm}
-\dfrac{\nu_1}{2}&\dfrac{4-\nu_1\nu_2}{2\mu_2}\\
\dfrac{\mu_1}{2}&\dfrac{\mu_1\nu_2}{2\mu_2}
\end{pmatrix}.
\end{equation}
Note that $\nu_j\geq 0$, $\mu_j> 0$, for $j=1,2$. The following observation will allows us to skip considering the case $\nu_1\nu_2=0$ within the arguments below.

\begin{lemma}\label{lem:trivial}
With the same notation as above, assume that $\nu_j=0$ for some $j$. Then $X\cong\P^1\times\P^1$.
\end{lemma}

\begin{proof}
We write the proof for $j=1$. Consider the ruled surface $S:=\pi_1^{-1}(\pi_1(\Gamma_2))\supset \Gamma_2$; since the family of deformations of $\Gamma_2$ is unsplit, $\Gamma_2$ must be a minimal section of $S$. Hence, $\nu_1=0$ implies that $S\cong\P^1\times\P^1$. In particular, every curve $\Gamma_i$ meeting $S$ is contained in $S$. Since, by hypothesis, $X$ has Picard number two and so it is rationally chain connected with respect to the curves in the families of deformations of $\Gamma_1$ and $\Gamma_2$, we conclude that $X=S$.
\end{proof}

We will make use later of the following Chern-Wu relation on smooth $\P^1$-fibrations:

\begin{lemma}\label{lem:codim2}
With the same notation as above, for every $j=1,2$ there exists $\Delta_j\in\Q$ such that $K_j^2=\Delta_jH_j^2$ modulo numerical equivalence.
\end{lemma}

\begin{proof}
We write the proof for $j=1$. Let $N^2(X_1)_{\Q}$ be the $\Q$-vector space of codimension $2$ cycles on $X_1$ modulo numerically equivalence. Following verbatim the proof of \cite[Proposition~2.3]{Wa},
we may show that $\pi_1^*(N^2(X_1)_{\Q})$ is isomorphic to $\Q$, and it is generated by $H_1^2$. Hence it suffices to show that $K_1^2\in\pi_1^*(N^2(X_1)_\Q)$.
In order to see this, we may take a curve $C\subset X_1$ and consider the surface $S=\pi_1^{-1}(C)$, which is a $\P^1$-fibration, hence a $\P^1$-bundle. Then the statement is equivalent to $(K_1)_{|S}^2=0$, which is known to be true for every $\P^1$-bundle over a curve.
\end{proof}

\begin{lemma}\label{lem:discriminant}
With the same notation as above, $\Delta_j<0$ for $j=1,2$, unless $X\cong\P^1\times\P^1$.
\end{lemma}

\begin{proof}
We write the proof for $j=1$.
Since $-K_1+\frac{\nu_1}{\mu_1}H_1$ is numerically proportional to $H_2$, it is nef and not big, hence its $(m+1)$th self-intersection is zero. If $\Delta_1=0$, then
$$
0=\left(-K_1+\frac{\nu_1}{\mu_1}H_1\right)^{m+1}=
-(m+1)(K_1 \cdot H_1)^m\left(\frac{\nu_1}{\mu_1}\right)^m.
$$
It follows that $\nu_1=0$, and we conclude by Lemma \ref{lem:trivial} above.

If $\Delta_1>0$, then using the Chern-Wu relation $K_1^2=\Delta_1H_1^2$, we may write:
$$
0=\left(-K_1+\frac{\nu_1}{\mu_1}H_1\right)^{m+1}=\dfrac{\big(\frac{\nu_1}{\mu_1}+\sqrt{\Delta_1}\big)^{m+1}- \big(\frac{\nu_1}{\mu_1}-\sqrt{\Delta_1}\big)^{m+1}}{2\sqrt{\Delta_1}}(-K_1\cdot H_1^m).
$$
Since $K_1H_1^m\neq 0$, we must have $\big(\frac{\nu_1}{\mu_1}+\sqrt{\Delta_1}\big)^{m+1}= \big(\frac{\nu_1}{\mu_1}-\sqrt{\Delta_1}\big)^{m+1}$. Since $\Delta_1>0$, this is not possible unless $m$ is odd and $\nu_1=0$, and we conclude by Lemma \ref{lem:trivial} again.
\end{proof}

In the case $\Delta_j<0$, let us set $b_j:=+\sqrt{-\Delta_j}$ and $z_j:=\frac{\nu_j}{\mu_j}+ib_j\in\C$, $j=1,2$.

\begin{lemma}\label{lem:235}
With the same notation as above, assume that $X\not\cong \P^1\times\P^1$. Then $(z_j)^{m+1}\in\R\mbox{ for }j=1,2$. Moreover  $m=2,3,5$.
\end{lemma}

\begin{proof}
As in the proof of Lemma \ref{lem:discriminant}, we may write:
$$
0=\left(-K_1+\frac{\nu_1}{\mu_1}H_1\right)^{m+1}=
\dfrac{\big(\frac{\nu_1}{\mu_1}+ib_1\big)^{m+1}- \big(\frac{\nu_1}{\mu_1}-ib_1\big)^{m+1}}{2ib_1}(-K_1\cdot H_1^m),
$$
hence, with the notation introduced above, this means that the imaginary part of $z_1^{m+1}$ is zero. The case $j=2$ is analogous.

Finally, the fact that $m=2,3,5$ follows from the fact that $\nu_1/\mu_1$ and $\Delta_1$ are rational numbers and that the argument of $z_1$ is $\pi/(m+1)$ (see \cite[Proposition~4.4]{MOS} for details).
\end{proof}

We may now conclude the classification of FT-manifolds of Picard number two.

\begin{proof}[Proof of Theorem \ref{thm:pic2}]
Assume $X$ is not $\P^1\times\P^1$, that is the complete flag of type $A_1\times A_1$. In particular, we may assume that  $\nu_1\nu_2\neq 0$, by Lemma \ref{lem:trivial}, and that $m=2,3$ or $5$, by Lemma \ref{lem:235}. We will show that
\begin{equation}\label{eq:cos}\nu_1\nu_2=4\cos^2\left(\frac{\pi}{m+1}\right).\end{equation}
Since this number equals $1,2$ or $3$, it follows that either $\nu_1$ or $\nu_2$ are equal to $1$. This provides a unisecant divisor in one of the structures of $\P^1$-fibration of $X$, hence $X$ is a $\P^1$-bundle and we may then conclude by \cite[Theorem 1.1]{Wa}.

In order to prove equality (\ref{eq:cos}), we consider the intersection number $K_1\cdot H_1^m\neq 0$ and write it, using the change of basis (\ref{eq:basechange}) and the Chern-Wu relation $K_2^2=b_2^2H_2^2$, in terms of $K_2\cdot H_2^m$:
$$
K_1\cdot H_1^m=\dfrac{\mu_1^m}{2^{m-1}\mu_2b_2}\mbox{im}(z_2^m)(K_2\cdot H_2^m).
$$
Using the same argument with indices exchanged we get:
$$
K_1\cdot H_1^m=\dfrac{\mu_1^{m-1}\mu_2^{m-1}}{4^{m-1}b_1b_2}\mbox{im}(z_1^m)\mbox{im}(z_2^m)(K_1\cdot H_1^m),
$$
therefore
$$
\dfrac{\mu_1^{m-1}\mu_2^{m-1}}{4^{m-1}b_1b_2}\|z_1\|^{m-1}b_1\|z_2\|^{m-1}b_2=1.
$$
Using that $\frac{\nu_j}{\mu_j}=\|z_j\|\cos(\pi/(m+1))$, this equality may be written as:
$$
\dfrac{\nu_1^{m-1}\nu_2^{m-1}}{4^{m-1}}=\cos^{2m-2}\left(\dfrac{\pi}{m+1}\right),
$$
which provides the desired equality (\ref{eq:cos}).
\end{proof}

As a consequence, we may  deduce severe restrictions on the coefficients of the Cartan matrix of an FT-manifold.

\begin{corollary}\label{cor:rank2} Let $X$ be an FT-manifold, and $M(X)=[m_{ij}]$ its Cartan matrix. Then:
\begin{itemize}
\item $m_{ii} = 2$,
\item $m_{ij}=0$ if and only if $m_{ji}=0$, and
\item if $m_{ij} \neq 0$, then $m_{ji}\in\Z^-$ and $m_{ij}m_{ji} =1,2$ or $3$.
\end{itemize}
In particular $M(X)$ is a generalized Cartan matrix in the sense of \cite[4.0]{Kac}.
\end{corollary}

\begin{proof} Any $2 \times 2$ principal submatrix is, by Proposition \ref{prop:fibers}, the Cartan matrix of an FT-manifold of Picard number $2$. These are, up to transposition, the ones corresponding to the manifolds appearing in Theorem \ref{thm:pic2}, which are
$$\begin{array}{ll}\vspace{0.2cm}
M(A_1 \times A_1)=\left(\begin{matrix}
2&0\\0&2
\end{matrix}\right),&
M(A_2)=\left(\begin{matrix}
2&-1\\-1&2
\end{matrix}\right),\\
M(B_2)=\left(\begin{matrix}
2&-1\\-2&2
\end{matrix}\right),&
M(G_2)=\left(\begin{matrix}
2&-1\\-3&2
\end{matrix}\right).
\end{array}
$$\end{proof}

This result allows us to associate, in an obvious way, a Dynkin diagram $\cD(X)$
with an FT-manifold $X$:

\begin{definition}
The {\it Dynkin diagram} $\cD(X)$ of $X$ is the graph having $n:=\rho_X$ nodes, such that the nodes in the $i$-th and $j$-th position are joined by $( -K_i \cdot \Gamma_j)( -K_j \cdot \Gamma_i)$ -- which is equal to $=0,1,2$ or $3$ -- edges. When two nodes are joined by multiple edges we write an arrow on them pointing to the node $j$ if $ -K_i \cdot \Gamma_j <  -K_j \cdot \Gamma_i$. The set of nodes of $\cD(X)$ will be identified with $D=\{1,\dots,n\}$. If $I \subset D$ is a non-empty subset, we will denote by $\cD(X)_I$ the subdiagram of $\cD(X)$ whose set of nodes is $I\subset D$.
\end{definition}

\begin{remark}\label{rem:fibers2}
From Proposition \ref{prop:fibers} it follows that, the Dynkin diagram $\cD(Z_I)$ of the fibers $Z_I$ of a contraction $\pi_I:X\to X_I$ equals  $\cD(X)_I$.
\end{remark}

\begin{proposition}\label{prop:reducible}
Let $X$ be an FT-manifold whose Dynkin diagram $\cD(X)$ can be written as the disjoint union of two subdiagrams $\cD(X)_I$ and $\cD(X)_J$, $I\sqcup J=D$ and let $\pi_I:X \to X_I$ and $\pi_J:X \to X_J$ be the contractions associated to $I$ and $J$. Then $X \simeq X_I \times X_J$ and, moreover, $X_I$ and $X_J$ are FT-manifolds whose Dynkin diagrams are $\cD(X)_I$ and $\cD(X)_J$, respectively.
\end{proposition}

\begin{proof} Without loss of generality we may assume that $\cD(X)_I $ is connected.

If $I=\{i\}$, then $-K_{i}$ is nef, and it is trivial on every ray $R_j$ with $j \not = i$. Hence it is a supporting divisor of the face $R_J$ and thus is semiample. For any fiber $Z_J$ of $\pi_J$, we denote by $W \to Z_J$ the base change of $\pi_i: X \to X_i$ by $Z_J$. This is a $\P^1$-bundle whose relative anticanonical divisor is semiample.
By \cite[Theorem 2.3]{Ya}, since $Z_J$ is simply connected, then $W\simeq \P^1 \times Z_J$. In particular the image  of $W$ in $X$ contains any extremal curve that it meets. Since $X$  is chain connected by curves of the form $\Gamma_1,\dots,\Gamma_n$ this implies that $W \to X$ is surjective, so,
according to \cite[Lemma 2.2]{Wa}, this implies that $X_J \simeq \P^1$. We then conclude that $X\simeq \P^1 \times X_i$ by Lemma \ref{lem:maxevid}.

Assume now that $\sharp(I)\ge 2$. Composing an elementary contraction of $X_I$ with $\pi_I$ we obtain a contraction of $X$ corresponding to a subset $I' \subset D$ given by $I'=I \sqcup \{j\}$, with $j \in J$. By Remark \ref{rem:fibers2}, the fibers $Z_{I'}$ of $\pi_{I'}$ are FT-manifolds whose Dynkin diagram is $\cD(X)_I \sqcup \{j\}$. By the first part of the proof they are products $Z_I \times \P^1$. It follows that any elementary contraction of $X_I$ has one dimensional fibers.

By Proposition \ref{prop:maxevid} we have that  $X \simeq X_I \times Z_J$. Since $\pi_I$ and $\pi_J$ are the projections to the factors we have  $X_I \simeq Z_J$ and $Z_I \simeq X_J$.
\end{proof}

\begin{corollary}\label{cor:reducible}
If Conjecture \ref{conj:CPconj} holds for FT-manifolds whose Dynkin diagram is connected, then it holds for every FT-manifold.
\end{corollary}

%%%%%%%%%%%%%%%%%%%%
% FT-manifolds
% Section: Main theorem
%%%%%%%%%%%%%%%%%%%%

\section{Classification of Dynkin diagrams of FT-manifolds}\label{sec:main}

%%%%%%%%%%%%%%%%%%%%%%%%%%%%%%%%%%%%%%%
This section is devoted to the proof of Theorem \ref{thm:picn}, whose goal is to show that $\cD(X)$ is the Dynkin diagram of a finite dimensional Lie algebra, or equivalently, in the language of \cite{Kac}, that $\cD(X)$ is {\it finite}.

In our proof we will need to discard the case in which $\cD(X)$ is {\it affine}, that is, in which there exists a vector $v=(v_1,\dots,v_n)\in\R^n$, $v_i>0$ for all $i$, satisfying that $M(X)\cdot v=0$ (cf. \cite[Corollary~4.3]{Kac}). The next proposition expresses the affineness of $\cD(X)$ in terms of curves contained in $X$:

\begin{proposition}\label{prop:CT-Cartan} Let $X$ be an FT-manifold, let $\cD(X)$ be its Dynkin diagram and assume that $\cD(X)$ is connected. Then
\begin{enumerate}
\item[1)] If  every proper subdiagram of $\cD(X)$ is finite, then $\cD(X)$ is finite or affine.
\item[2)] If $\cD(X)$ is affine then there exists  an irreducible rational curve $B \subset X$  such that $B \equiv \sum_1^n m_i \Gamma_i$, with $m_i \in \mathbb Z_{>0}$, and $K_i \cdot B=0$ for every $i=1, \dots, n$.

\end{enumerate}
\end{proposition}

\begin{proof}
The first assertion follows from \cite[Proposition~4.7]{Kac}, items a) and b). For the second, note that, being $\cD(X)$ affine, there exists a linear combination $\Gamma=\sum_1^n m_i \Gamma_i$, $m_i\in\R_{>0}$ for all $i$, satisfying that $K_i\cdot \Gamma=0$. Since the kernel of the matrix $M(X)$ is generated by rational cycles, we may assume that $m_i\in \Q_{>0}$ for all $i$ and, clearing denominators, we may also assume that $m_i\in\Z_{>0}$ for all $i$. Finally, being $T_X$  nef, every rational tree in $X$ is smoothable by \cite[Proposition~4.24]{De}, therefore the cycle $\Gamma=\sum_1^n m_i \Gamma_i$ is numerically equivalent to an irreducible rational curve $B$.
\end{proof}

Our strategy to prove that the Dynkin diagram of an FT-manifold $X$ is finite consists of showing that a rational curve such as the one in Proposition \ref{prop:CT-Cartan}, 2) does not exist. In order to do this, we will need two auxiliary lemmata.

\begin{lemma}\label{lem:product}
Let $M$ be an irreducible variety and $f:\P^1 \times M \rightarrow X$ be a morphism. Let $b\in \P^1$ be a point and $j:M \to \P^1 \times M$ be the corresponding inclusion $j(z)=(b,z)$. Given an elementary contraction $\pi_i:X\to X_i$, set $g:=\pi_i \circ f$ and consider the following fiber products:
$$
\xymatrix{M\times_{X_i}X\ar[r]\ar[d]&(\P^1\times M)\times_{X_i}X\ar[r]\ar[d]&X\ar[d]^{\pi_i}\\
M\ar[r]^j&\P^1\times M\ar[r]^g\ar[ru]^{f}&X_i}
$$
If $f^*K_i \cdot (\P^1\times\{z\})=0$, then $(\P^1\times M)\times_{X_i}X\cong\P^1\times (M\times_{X_i}X)$.
\end{lemma}
\begin{proof}
Note that the morphism $f:\P^1\times M\to X$ factors via $(\P^1\times M)\times_{X_i}X$, providing a section of the natural map  $(\P^1\times M)\times_{X_i}X\to \P^1\times M$ and giving $(\P^1\times M)\times_{X_i}X$ the structure of a $\P^1$-bundle over $\P^1\times M$. More precisely, it is the projectivization of a rank two bundle $\cE$ on $\P^1\times M$ appearing as an extension:
$$
0\rightarrow\cO(f^*K_i)\longrightarrow\cE\longrightarrow\cO\rightarrow 0
$$
Now, since $f^*K_i$ has intersection zero with the fibers of the projection $p:\P^1\times M\to M$, then $\cE$ is trivial on these fibers and it follows that $\cE=p^*j^*\cE$, from which our statement follows.
\end{proof}

\begin{lemma}\label{lem:finite} Let $M$ be an irreducible projective variety and $f: \P^1 \times M \to X$ be a morphism such that  $f(\P^1 \times \{z\})$ is not contracted by $\pi_i$. If  $\pi_i \circ f$ is of fiber type then, for any $b \in \P^1$, the morphism $(\pi_i \circ f)|_{\{b\} \times M}$ is of fiber type.
\end{lemma}

\begin{proof} Denote by $p_1$ and $p_2$ the projections of $\P^1 \times M$ onto the factors.
Let $(b,z)$ be a point of $\P^1 \times M$, and $C$ an irreducible curve passing through $(b,z)$ contracted by
$\pi_i \circ f$. As a $1$-cycle, $C$ is numerically equivalent in $\P^1 \times M$ to a combination with nonnegative coefficients $\lambda \Gamma_b + \mu \P^1_z$, where $\Gamma_b$ is a curve contained in  $p_1^{-1}(b)$ and $\P^1_z:=p_2^{-1}(z)$.
The curve $C$ is contracted by $\pi_i \circ f$, therefore  $[f_*C] \in R_i$, so, by the extremality of $R_i$,
$[\mu f_*\P^1_z] \in R_i$. This contradicts our assumptions, unless $\mu=0$.
We can thus conclude that  $C$ is contained in $p_1^{-1}(b)$.
\end{proof}

Next we introduce, for FT-manifolds, a  construction similar to the so-called {\it basic construction} in \cite[p.~562]{Ke}.

\begin{construction}\label{con:prod}
Let $X$ be an FT-manifold, $f_0:V_0 \to X$ a morphism from a smooth variety $V_0$, and $i_1, i_2, \dots, i_\ell \in \{1,2,\dots, n\}$
a list of (not necessarily distinct) indices.
Starting from this data we construct a pair $(V_k,f_k)$, $k=1,\dots,\ell$, recursively in the following way:
given $f_{k-1}(V_{k-1})\subset X$  we define $V_k$ to be the fiber product $V_{k-1}\times_{X_{i_k}}X$,
and $f_k$ to be the natural map from $V_k$ to $X$:
$$
\xymatrix{V_k\ar[rr]^{f_k}\ar[d]^{\pi_k}&&X\ar[d]^{\pi_{i_k}}\\V_{k-1}\ %ar@/^/[u]^{\sigma_k}
\ar[rr]_{g_{k-1}}\ar[rru]^{f_{k-1}}\ar[rr]&&X_{i_k}}
$$
Note that, if $m=\dim X -\dim V_0$ we can choose  indices $i_1, i_2, \dots, i_m$ so that, at each step, $f_{k-1}(V_{k-1})$ is a divisor in $f_k(V_k)$. The process stops at step $m$, and $f_{m}(V_m) = X$.
\end{construction}

\begin{proposition}\label{prop:try}
There is no  rational curve $B \subset X$ whose numerical class lies in the interior of $\NE(X)$ such that $K_i \cdot B=0$ for every $i=1, \dots, n$.
\end{proposition}

\begin{proof}
Assume that, on the contrary, there exists a curve $B$ satisyfying those properties. We perform Construction \ref{con:prod}, starting from a smooth point $p \in B$ until $m= \dim V_m= \dim X$,  denoting by $i_1,i_2, \dots, i_m$ the indices corresponding to the curves used in the construction. Using the same list of indices $i_1,i_2, \dots, i_{m-1}$ we perform again Construction
\ref{con:prod}, starting now from $V'_0=\P^1$ and $f'_0$ the normalization $f'_0:\P^1 \to B$, so that we obtain pairs $(V_k',f'_k)$.

Denoting $b:={f'_0}^{-1}(p)$, and by $j_k:V_k \to \P^1 \times V_k$ the inclusion $j_k(v)=(b,v)$, we claim that the following holds for every $k$:  $V_k'\simeq \P^1 \times V_k$, and, under this identification, $f_k' \circ j_k= f_k$.

The statement is clearly true for $k=0$, so we assume by induction that it holds for $k=\ell-1$.
In particular, we have $g_{\ell-1}' \circ j_{\ell-1}= g_{\ell-1}$ and, since $K_{i_{\ell}}\cdot B=0$,  Lemma \ref{lem:product}, tells us that
 $$V'_\ell \simeq (\P^1 \times V_{\ell-1})\times_{X_{i_\ell}}X\cong \P^1 \times (V_{\ell-1}\times_{X_{i_\ell}}X)\cong \P^1 \times V_{\ell}.$$
Finally, since both the left and the right squares of the diagram
$$
\xymatrix@C=35pt{V_\ell \ar[d]_{\pi_\ell}\ar[r]^(.38){j_\ell} &\P^1 \times V_\ell \ar[d]_{\pi'_\ell}\ar[r]^{f'_\ell}&X \ar[d]^{\pi_{i_\ell}}\\
V_{\ell-1}  {\ar@/_1pc/[rr]_{g_{\ell-1}} } \ar[r]^(.45){j_{\ell-1}}&\P^1 \times V_{\ell-1}\ar[r]^{g'_{\ell-1}}& X_{i_\ell}}
$$
are cartesian, also the outer square is, hence $f_\ell' \circ j_\ell= f_\ell$, and the claim is proved.

In particular, taking $k=m-1$ we get $V'_{m-1} = \P^1 \times V_{m-1}$. Since by hypothesis $B$ is not contracted by $\pi_i$ for any $i$, then by Lemma \ref{lem:finite} the morphism $\pi_{i_m} \circ f'_{m-1}:V'_{m-1} \to X_{i_m}$ is generically finite, contradicting that $\dim V'_{m-1}= \dim X$.
\end{proof}

We can now prove the main result of this paper:

\begin{proof}[Proof of Theorem \ref{thm:picn}]
By Proposition \ref{prop:reducible} we can assume that $\cD(X)$ is connected.
We prove the statement by induction on the number of nodes of $\cD(X)$. The result
 is true for $n=2$ by Theorem \ref{thm:pic2}, hence we may assume that $n\geq 3$ and that the statement holds for FT-manifolds of Picard number $\le n-1$.
By induction every subdiagram of $\cD(X)$ is finite, hence,
by Proposition \ref{prop:CT-Cartan} 1), we may assert that  $\cD(X)$ is either finite  or affine.

In the latter case, item 2) in the same Proposition provides an irreducible rational curve $B \subset X$  such that $B \equiv \sum_1^n m_i \Gamma_i$, with $m_i \in \mathbb Z_{>0}$, and $K_i \cdot B=0$ for every $i=1, \dots, n$. This contradicts Proposition \ref{prop:try}.\end{proof}

\section{Chains of rational curves on FT-manifolds}\label{sec:chains}

Loci of chains of minimal rational curves in the extremal rays of an FT-manifold $X$ admit a very peculiar description in terms of the corresponding elementary contractions. Given a subset $Y \subset X$ we define $\Ch(i_1 \dots i_m)(Y)$ to be the set of points $x \in X$ such that there exist curves $\Gamma_{i_1}, \dots, \Gamma_{i_l}$ satisfying $\Gamma_{i_j} \cap \Gamma_{i_{j+1}} \not = \emptyset$, $\Gamma_{i_1} \cap Y \not = \emptyset$ and $x \in \Gamma_{i_m}$, that is:
$$\Ch(i_1 \dots i_m)(Y)=\pi_{i_m}^{-1} (\pi_{i_m}( \dots (\pi_{i_1}^{-1} (\pi_{i_1}(Y))))).$$

Let us first consider the case in which $X=F$ is a complete flag manifold, for which we will study the loci of chains of curves $C_i$, $i=1,\dots,n$.
\begin{remark}
If $X=F=G/B$ is a complete flag manifold, chains of rational curves as above are related to elements of the Weyl group $W$ of $G$. In fact, the loci $\Ch(i_1 \dots i_m)(y)\subset F$ are {\it Schubert varieties}, that is the closures of the subsets of the form $BwB/B$, $w\in W$. In order to make this relation more explicit we need the description of $W$ as a Coxeter group, whose definition we introduce below. For further details we refer the reader to  \cite[Section 2]{Ke}.
\end{remark}

\begin{definition}\label{def:coxgroup}
A {\em Coxeter group} is  a group admitting a  presentation of the form $$\langle s_1, s_2, \dots, s_n \,|\, (s_is_j)^{a_{ij}} =1 \rangle, \mbox{ where }a_{ii} = 1,\mbox{ and }a_{ij} = a_{ji} \ge 2 \mbox{ for }i \not = j.$$ Denoting by $r_k(s_i,s_j)$ the word of length $k$ which contains $s_i$ in every odd position and $s_j$ in every even position and by $l_k(s_i,s_j)=r_k(s_j,s_i)$, an equivalent presentation of the group is given by:
$$\langle s_1, s_2, \dots, s_n \,|\, s_i^2 =1, r_{a_{ij}}(s_i,s_j)=l_{a_{ij}}(s_i,s_j) \rangle.$$
A {\em reduced word} in the group is a word of minimal length among those presenting an element. The {\em length} of an element  is the length of a reduced word presenting that element.
\end{definition}

\begin{remark}
The Weyl group $W$  of a semisimple Lie algebra $\fg$ is the Coxeter group whose exponents $a_{ij}$ are obtained from the Dynkin diagram of $\fg$ via the formulas
\begin{equation}\label{eq:relcoeff} 4 \cos^2\left(\dfrac{\pi}{a_{ij}}\right)= \sharp\{\text{edges joining $i$ and $j$}\}.
\end{equation}
Given a reduced word $w=s_{i_1}\dots s_{i_\ell}$ the corresponding Schubert variety $\overline{BwB/B}$
equals $\Ch(i_1 \dots i_\ell)(y)$ and the dimension of $\overline{BwB/B}$ is equal to the length of $w$.
\end{remark}

\begin{definition}
Let $W$ be a Coxeter group as in Definition \ref{def:coxgroup}; the {\em Coxeter monoid} associated to $W$ is the monoid $W'$ given by the presentation
$$ W':=\langle s_1, s_2, \dots, s_n \,|\, s_i^2 =s_i, r_{a_{ij}}(i,j)=l_{a_{ij}}(i,j) \rangle.$$
The {\em length} of an element of $W'$ is defined as in \ref{def:coxgroup}.
\end{definition}

In \cite[Theorem 1]{Ts} it has been proved that there is a bijection between reduced words in $W$ and $W'$, where two reduced words in $W$ are equal if and only if the corresponding words in $W'$ are equal.
In particular, if $W$ is finite, then $W'$ is finite, and the maximum length of a reduced word in $W'$ equals the maximum length of a reduced word in $W$, which coincides with the dimension of $F$. So we have

\begin{lemma}\label{lem:dimflag} If $F=G/B$ is a complete flag manifold, then $\dim F$ equals the length of the longest word of the monoid $W'$ associated with the Weyl group $W$ of $G$.
\end{lemma}

Let us extend the previous construction to the case in which $X$ is an FT-manifold. Since $X$ has been endowed with a Dynkin diagram, we may associate with $X$ a Coxeter monoid $W'(X)$. However it is not clear a priori that Lemma \ref{lem:dimflag} holds
in this case. In turn we choose a general point $y\in X$ and consider the monoid $\ch(X)$ generated by the symbols $1,2,\dots, n$ subject to the  following relations:
$$i_1i_2\dots i_l=j_1j_2\dots j_k \quad \text{if and only if} \quad \Ch(i_1i_2 \dots i_l)(y)=\Ch(j_1j_2 \dots j_k)(y).$$
From the definition it is clear that for every $i$ we have $ii=i$; moreover,
it is easy to check, considering the fibers of  $\pi_{\{i,j\}}$, that we also have
$$ r_{a_{ij}}(i,j)=l_{a_{ij}}(i,j),$$
where the $a_{ij}$ are obtained from the Dynkin diagram of $X$ via the formulas in (\ref{eq:relcoeff}).
In other words, $\ch(X)$ is a quotient of $W'(X)$.

Observing that if $w$ is a reduced word in $\ch(X)$ then $\dim \Ch(w)(y)= l(w)$ (where the length $l(w)$ is defined as in \ref{def:coxgroup}), we have the following

\begin{corollary}\label{cor:expecteddim}
Let $X$ be an FT-manifold with Dynkin diagram $\cD(X)$. If $\cD(X)$ is finite and $G/B$ is the variety of flags corresponding to $\cD(X)$, then $\dim X \le \dim G/B$.
\end{corollary}

\begin{proof}
It is enough to remark that $\dim X$ equals the length of the longest element in $\ch(X)$.
\end{proof}

%%%%%%%%%%%%%%%%%%%%
% FT-manifolds
% Section: An
%%%%%%%%%%%%%%%%%%%%

\section{CP Conjecture for FT-manifolds of type $A_n$}\label{sec:An}

Before starting with the case $A_n$, let us consider the following general situation. Given an FT-manifold $X$ of Picard number $n$, with Dynkin diagram $\cD(X)$, let us consider a semisimple Lie algebra $\fg$ with the same Dynkin diagram, and choose a Cartan decomposition for it, with root system $\Phi$ and simple roots $\alpha_i$, $i\in D=\{1,\dots,n\}$.
Given a subset $I\subsetneq D$, every fiber $Z_I$ of the corresponding contraction $\pi^I:X\to X^I$ is an FT-manifold with Dynkin diagram $\cD(X)_{D\setminus I}$. Assume that $Z_I$ is homogeneous. Then it is a quotient of a Lie group whose set of positive roots may be identified with $$\Phi^{I+}=\left\{\left.\sum_{i=1}^nr_i\alpha_i\in\Phi^+\right|r_i=0\mbox{ for }i\in I\right\}.$$
We may then use the filtration of $T_{Z^I}$ described in  Construction \ref{const:foliation}, that
provides a sequence $\Psi^I_1\subsetneq\dots\subsetneq\Psi^I_{\dim (Z_I)}={\Phi^I}^+$ of admissible sets allowing us to construct, recursively, vector subbundles $T_{\Psi^I_k}\subset T_{Z^I}$: at step $k>\sharp(D\setminus I)$ this bundle is the inverse image into $T_{Z^I}$ of the image of a morphism
$$
N_k:T_{\Psi^I_j}\otimes (T_\ell)_{|{Z^I}} \longrightarrow T_{Z^I}/T_{\Psi^I_{k-1}}, \mbox{ for some }j<k,\ell\in D\setminus I.
$$
Note also that, setting $\sum_{\beta\in\Phi^{I+}}\beta=:\sum_{i\in D\setminus I}m_i\alpha_i$, it follows that
$$
K_{Z^I}=\sum_{i\in D\setminus I}m_i{K_i}_{|Z^I}.
$$
The next statement tells us that the above construction on $Z_I$ can be made globally in $X$, and used to get numerical properties of the canonical divisor of $X^I$.

\begin{proposition}\label{prop:expint}
Let $X$ be an FT-manifold of Picard number $n$, $I\subsetneq D=\{1,\dots,n\}$, let $F=G/B$ be the complete flag manifold associated with the Dynkin diagram $\cD(X)$ and
assume that CP Conjecture holds for FT-manifolds with Dynkin diagram $\cD(X)_{D\setminus I}$.
Then, with the same notation as above:
 \begin{itemize}
   \item $K^I\equiv\sum_{i\in D\setminus I}m_i{K_i}$, and
   \item $K_{{X^I}}\cdot \pi^I_*(\Gamma_i)=K_{F(I)}\cdot p^I_*(C_i)\quad\mbox{for every }i\in D$.
 \end{itemize}
\end{proposition}

\begin{proof}
Let $Z^I$ be a fiber of $\pi^I$, that is, by hypothesis, isomorphic to the complete flag manifold associated to the Dynkin diagram $\cD(X)_{D\setminus I}$.

By induction, assume that, at step $k$, the vector bundles $T_{\Psi^I_{r}}$ are restrictions to $Z^I$ of vector subbundles $\cV_r\subset T_X$ in $X$ for every $r<k$. Since the relative tangent bundle $T^I$, whose leaves are the fibers $Z^I$, is integrable, it follows that the morphisms $N_k$ defined on every fiber $Z^I$ glue together into the O'Neill tensor $N_k$ defined globally by
$$
N_k:\cV_j\otimes T_\ell \longrightarrow T^I/{\cV_{k-1}}.
$$

By induction we may also assume that $\cV_j/\cV_{j-1}$ is a tensor product combination of the $T_r$'s. Since, following Construction \ref{const:foliation}, the image of $N_k$ is isomorphic to $\cV_j/\cV_{j-1}\otimes T_\ell$, we conclude that $\cV_{k}/\cV_{k-1}$ is a tensor product combination of the $T_r$'s. By construction, the coefficients of the combination can be computed by restricting to any fiber $Z^I$.

Summing up, the coefficients of ${K^I}\equiv -c_1\left(\bigoplus_k\cV_{k}/\cV_{k-1}\right)$ as a linear combination of the $K_i$'s, $i\notin I$ may  be computed by restricting to any $Z_I$. This proves the first assertion. The second is then obtained by intersecting with $\Gamma_i$ and $C_i$ on $X$ and $F$, respectively, and applying adjunction and projection formulas.
\end{proof}

Let $\cD$ be a connected finite Dynkin diagram, corresponding to a simple Lie group $G$. With the same notation as in Section \ref{sec:nothom}, we will denote by $N(i)\subset D$ the set of neighboring nodes of the node $i\in D$, i.e. the set of nodes $j$ linked to $i$ in $\cD$.

Consider now a subset $I\subset D$ and a node $i\in I$. Following \cite{LM}, we say that  $i$ is an {\it exposed short node} for $I$ if the connected component containing $i$ of the subdiagram of $\cD$ supported on $D\setminus(I\setminus\{i\})$ contains an arrow pointing towards $i$.

Given a subset $I\subset D$ and a node $i\in I$, we will now set:
$$ J:=I\setminus\{i\},\quad I':=I\cup N(i),\quad  J':=J\cup N(i)$$

Consider the line bundle $L \in {\rm Pic}(F(I)) \simeq H^2(F(I), \Z)$ corresponding to $\sum_{i \in  I}\lambda_i$. Then $L$ is very ample and the embedding $F(I) \subset \P^N$ defined by the complete linear system $|L|$ is called  {\it minimal homogeneous}. With respect to this embedding, the coroots $\check{\alpha}_i \in H_2(F(I), \Z)$ are lines on $F(I)$ and \cite[Theorem~4.3]{LM} can be reformulated as follows:

\begin{theorem}[{\cite[Theorem~4.3]{LM}}]\label{thm:linesonflags} Let $X\subset \P^N$ be the minimal homogeneous embedding of $X:=F(I)$, and assume that $i$ is not an exposed short node for $I$. Then:
\begin{enumerate}
\item  the space of lines of class $\check{\alpha}_i \in H_2(F(I), \Z)$ is isomorphic to $F(J')$;
\item the following natural double fibration gives the universal family of lines of class $\check{\alpha}_i $:
\[\xymatrix@=20pt{
& F(I') \ar[dl]_{}  \ar[dr]^{}  &  \\
F(I)&& F(J')  . \\
} \]
\end{enumerate}
\end{theorem}

The following Proposition shows how to use Theorem \ref{thm:linesonflags} to reconstruct
inductively an FT-manifold  using family of lines:

\begin{proposition}\label{prop:onestep}
Let $X$ be an FT-manifold with Picard number $n$ and connected Dynkin diagram $\cD(X)$. Assume that $i\in I\subset D$ satisfy that:
\begin{enumerate}
\item $i$ is not an exposed short node for $I$;
\item $\sharp(I')=\sharp(I)+1 \ge 3$.
\end{enumerate}
Assume furthermore that CP Conjecture holds for FT-manifolds whose Dynkin diagram
is a subdiagram of $\cD(X)$.
Then, if $X^{I}\cong F(I)$ it follows that $X^{I'}\cong F(I')$.
\end{proposition}

\begin{proof}
By induction, for every $\emptyset \not = \bar J \subset \bar I \subset D$ the fibers of $\pi^{\bar I, \bar J}:X^{\bar I} \to X^{\bar J}$ are isomorphic to the fibers of $p^{\bar I, \bar J}:F(\bar I) \to F(\bar J)$. This, together with the assumption $X^{I}\cong F(I)$, implies that $\dim X = \dim F(D)$, and, in turn that $\dim X^{\bar I} = \dim F(\bar I)$ for every $\bar I \subset D$.

In particular the fibers of $\pi^{I,J}:X^{I} \to X^{J}$ are isomorphic to the fibers of
$p^{I,J}:F(I) \to F(J)$, which have Picard number one. The curves $\pi^{I}_{\ast}(\Gamma_i)$
are contained in these fibers and, by Proposition \ref{prop:expint}, we see that $(-K_{X^{I}}) \cdot \pi^{I}_{\ast}(\Gamma_i)=(-K_{F(I)}) \cdot p^{I}_{\ast}(C_i)$.
This implies that these curves are lines of class $\check{\alpha}_i$ on $X^{I}$.
Thus, $\pi^{I',J'}:X^{I'} \to X^{J'}$ can be regarded as a family of lines
of class $\check{\alpha}_i$ on $X^{I}$.
By the universal property of the Hilbert scheme, this family is obtained by a base change from the universal
family of lines of class $\check{\alpha}_i$ on $X^{I}$,
hence, by Theorem \ref{thm:linesonflags}, we have the following commutative diagram:

\[\xymatrix@=35pt{
X^{I'}  {\ar@/^1pc/[rr]^{\pi^{I'\!,I}} }  \ar[r]_(.45){\tilde{h}} \ar[d]_{\pi^{I'\!,J'}}& F(I') \ar[d]_{p^{I'\!,J'}}  \ar[r]_{p^{I'\!,I}}  &X^{I}  \ar[d]^{\pi^{I,J}}  \\
X^{J'}  {\ar@/_1pc/[rr]_{\pi^{J'\!,J}} } \ar[r]^(.45)h  & F(J') \ar[r]^{p^{J'\!,J}}& {X^{J}}  \\
} \]

We claim that $h$ (and hence $\tilde h$) is a finite surjective map; by dimensional reasons it is enough to show that $h$ is finite. Assume by contradiction that there exists  a curve $C \subset X^{J'}$ which
is contracted by $h$. By the commutativity of the diagram $C$ is contained in a fiber of $\pi^{J',J}$.
Denote by $x$ the point $\pi^{J',J}(C)$ and consider the restricted diagram
\[\xymatrix@=35pt{
Z^{I'\!,J}  \ar[r]^(.45){\tilde{h}} \ar[d]_{\pi^{I'\!,J'}}& Z^{I'\!,J}  \ar[d]^{p^{I'\!,J'}}  \ar[r]^{ p^{I'\!,I}}  & Z^{I,J}  \ar[d]^{\pi^{I,J}}  \\
Z^{J'\!,J}   \ar[r]^(.45)h  & Z^{J'\!,J}   \ar[r]^{p^{J'\!,J}}& x \\
} \]

By assumption (2) $Z^{J'\!,J}$ has Picard number one, hence, $h|_{Z^{J'\!,J} }$ is constant.
Then $\tilde h(Z^{I'\!,J})$ is contained in a fiber of $p^{I'\!,J'}$, which has dimension one.
But, since the connected component containing $i$ of the subdiagram of $\cD(X)$ supported on $D\setminus J$ has at least two nodes,  $\dim Z^{I,J} > 1$, a contradiction which proves that  $\tilde h$ is finite and surjective.

By \cite[Main Theorem]{Lau}, for any fiber $Z^{I'\!,J}$ of  $\pi^{I'\!,J}$ the restriction ${\tilde h}|_{Z^{I'\!,J} }$ is an isomorphism (by assumption (2) the Picard number of $Z^{I'\!,J} $ is at least two). It follows that $\tilde h$ is bijective, hence an isomorphism.
\end{proof}

In the next proposition we will number the nodes of every connected finite Dynkin
as in \cite[p.~58]{H}.

\begin{proposition}\label{prop:recurs}
Let $X$ be an FT-manifold with connected Dynkin diagram $\cD(X)$, and consider, for every possible type of finite connected Dynkin diagram the subset $I\subset D$
 provided in the following table:
 \renewcommand{\arraystretch}{1.30}
\begin{center}
\vspace{0.2cm}
\begin{tabular}{|c||c|c|c|c|c|c|}
\hline
$\cD$&$A_n$&$B_n$&$C_n$&$D_n$&$E_n$&$F_4$\\\hline\hline
$I$&$\{1,n\}$&$\{1,n\}$&$\{1,n\}$&$\{1,n\}$&$\{2,n\}$&$\{1,2\}$\\\hline
\end{tabular}
\vspace{0.2cm}
\end{center}
Assume furthermore that  CP Conjecture holds for FT-manifolds whose Dynkin diagram
is a subdiagram of $\cD(X)$. Then, if $X^{I}\cong F(I)$, it follows that $X\cong F(D)$.
\end{proposition}
\begin{proof}
It is enough to show that there exists a sequence $\{(I_k,i_k),\,\,k=1,\dots,r \}$, $I_k\subset D$, $i_k\in I_k$, with $I_1=I$, satisfying that
\begin{itemize}
\item $i_k\in I_k\subset D$ satisfy the hypotheses of Proposition \ref{prop:onestep},
\item $I_{k+1}=I_{k}\cup\{i_{k+1}\}=I_{k}\cup N(i_k)$, and
\item $I_r=D$.
\end{itemize}
The existence of this sequence is an easy exercise. For instance, in case $A_n$ the sequence is given by $I_k=\{1,\dots,k,n\}, i_k=k$. \end{proof}

Finally, the next statement, together with the previous proposition, allows us to prove Theorem \ref{thm:An}.

\begin{proposition}\label{prop:step1An}
Let $X$ be an FT-manifold with Dynkin diagram $A_n$, with $n \ge 2$. Then $X^{\{1,n\}}\cong \P(T_{\P^n})\cong F(\{1,n\})$.
\end{proposition}

\begin{proof}
We proceed by induction on $n$. If $n=2$, then our claim follows from Theorem \ref{thm:pic2}.
Hence we may assume that $n \geq 3$ and that the statement holds for FT-manifolds of type $A_k$ if $k < n$.

By Remark \ref{rem:fibers2} together with our induction hypothesis, for every proper subset $I$ the fibers $Z_I$ of the contraction $\pi_I:X \to X_I$ are complete flag manifolds with Dynkin diagram $\cD(Z_I)=\cD(X)_I$; the connected components of
$\cD(Z_I)$ are clearly of type $A_k$ with $k<n$.
This implies that, for every $ \emptyset \neq I \subset J \subset \{1, \dots, n\}$, the fibers of the contraction %$\pi_{I,J}: X_I \rightarrow X_{J}$ and
$\pi^{J,I}: X^J\rightarrow X^{I}$ are isomorphic to the fibers of the contraction of rational homogeneous spaces $F(J)\rightarrow F(I)$.

In particular,
$\pi^{\{1,n\},1}: X^{\{1,n\}} \rightarrow X^1$ and $\pi^{\{1,n\},n}: X^{\{1,n\}} \rightarrow X^n$ are $\P^{n-1}$-fibrations.  Since $\dim X^{\{1,n\}}\leq 2n-1$ by Corollary \ref{cor:expecteddim}, we obtain:
\begin{eqnarray} \nonumber
\dim {X^{1}} + \dim {X^{n}} &\geq& \dim X^{\{1, n\}} \\ \nonumber
&=&(\dim X^1+\dim X^n) + 2(n-1)-\dim X^{\{1,n\}} \\ \nonumber
&\geq& (\dim X^1+\dim X^n) -1. \nonumber
\end{eqnarray}
Moreover, if $\dim X^1 + \dim X^n = \dim X^{\{1, n\}}$, then $X^1$ and $X^n$ are isomorphic to $\P^{n-1}$ (see \cite[Section 5]{OW}). Then $-K_{X^1}\cdot\pi^1_*(\Gamma_1)\neq n+1$, contradicting Proposition \ref{prop:expint}. Thus we have
\begin{eqnarray} \nonumber
\dim X^1 + \dim X^n = \dim X^{\{1, n\}} +1.
\end{eqnarray}
By \cite[Theorem~2]{OW}, we find that $X^{\{1, n\}}=\P(T_{\P^n})$ and $\pi^{\{1,n\},1}:X^{\{1, n\}} \rightarrow X^1=\P^n$, $\pi^{\{1,n\},n}:X^{\{1, n\}}\rightarrow X^n =\P^n$ are the canonical projections.  It is important to remark that we are allowed to use \cite{OW} because the term ``$\P^r$-bundle'' in that paper refers to what in this paper is called  $\P^r$-fibration.
\end{proof}

%%%%%%%%%%%%%%%%%
%
%Section:Bibliography
%%%%%%%%%%%%%%%%%


\begin{thebibliography}{MOS}

\bibitem[CO]{CO1} Chierici, E. and Occhetta, G., {\it The cone of curves of {F}ano varieties of coindex four},
{Internat. J. Math.},  {\bf 17} (2006), {1195--1221}.

\bibitem[CP]{CP} Campana, F. and Peternell, T.,
 {\it Projective manifolds whose tangent bundles are numerically effective},
 {Math. Ann.},  {\bf 289} (1991), no.1, 169--187.


\bibitem[CP2]{CP2} Campana, F. and Peternell, T.,
{\it {$4$}-folds with numerically effective tangent bundles and second {B}etti numbers greater than one}, {Manuscripta Math.}, {\bf 79} (1993), no.3-4, 225--238.

\bibitem[Ca]{Ca} Casagrande, C., {\it Quasi-elementary contractions of {F}ano manifolds},
{Compos. Math.},  {\bf 144} (2008), {1429--1460}.

\bibitem[D]{De} Debarre, O., {\it Higher-dimensional algebraic geometry},
 Universitext. Springer-Verlag, New York (2001).

\bibitem[DPS]{DPS} Demailly, J.P., Peternell, T. and Schneider, M. {\it Compact complex manifolds with numerically effective tangent bundles}, J. Algebraic Geom., {\bf 3} (1994), no. 2, 295--345.

\bibitem[GHS]{GHS} Graber, T., Harris, J. and Starr, J.
{\it Families of rationally connected varieties.}
J. Amer. Math. Soc. {\bf 16} (2003), no. 1, 57--67.

\bibitem[H]{H} Humphreys, J.E. {\it Introduction to Lie Algebras and Representation Theory}, Graduate Texts in Mathematics, 9. Springer-Verlag, New York, 1978.

\bibitem[Hw]{Hw} Hwang, J.-M.,
 {\it Rigidity of rational homogeneous spaces},
in {International {C}ongress of {M}athematicians. {V}ol. {II}}, Eur. Math. Soc., Z\"urich, (2006), {613--626}.

\bibitem[K]{Kac} Kac, V.G. {\it Infinite dimensional Lie algebras: an introduction}, Progress in Mathematics, vol. 44. Birkhauser, Boston, 1983.

\bibitem[Ke]{Ke} Kempf, G.R., {\it Linear systems on homogeneous spaces},
{Ann. of Math. (2)}, {\bf 103}, (1976), no. 3, 557--591.

\bibitem[LM]{LM} J. M. Landsberg, L. Manivel, {\it On the projective geometry of rational homogeneous varieties},
Comment. Math. Helv. {\bf 78} (2003), no. 1, 65-100.

\bibitem[La]{Lau} Lau, C.-H., {\it Holomorphic maps from rational homogeneous spaces onto projective manifolds}, {J. Algebraic Geom.}, {\bf 18}, (2009), no. 2, 223--256.

\bibitem[Mok]{Mk} Mok, N.,
{\it On {F}ano manifolds with nef tangent bundles admitting $1$-dimensional varieties of minimal rational tangents},  {Trans. Amer. Math. Soc.}, {\bf 354} (2002), no.7, 2639--2658.

\bibitem[Mor]{Mo} Mori, S., {\it Projective manifolds with ample tangent bundles},
 {Ann. of Math. (2)}, {\bf 110} (1979), no. 3, 593--606.

\bibitem[MOS]{MOS} Mu\~noz, R., Occhetta, G. and Sol\'a Conde, L.E. {\it On rank $2$ vector bundles on Fano manifolds}, Kyoto Journal of Math., {\bf 54} (2014), no.1 167--197.

\bibitem[OW]{OW}
Occhetta, G. and Wi{\'s}niewski, J.A.,
{\it On {E}uler-{J}aczewski sequence and {R}emmert-{V}an de {V}en problem for toric varieties},
{Math. Z.}, {\bf 241}, (2002), no.1, 35--44.

\bibitem[SW]{SW} Sol\'a Conde, L.E. and Wi\'sniewski, J.A. {\it On manifolds whose tangent bundle is big and $1$-ample}, Proc. London Math. Soc. (3) {\bf 89} (2004), no. 2, 273--290.

\bibitem[T]{Ts} Tsaranov, S. V., {\it Representation and classification of {C}oxeter monoids},
European J. Combin., {\bf 11}, (1990), no. 2, 189--204.

\bibitem[Y]{Ya} Yasutake, K.,  {\it On projective space bundles with nef normalized tautological divisor.}
arXiv:1104.5084v2.

\bibitem[W1]{Wa} Watanabe, K., {\it {${\mathbb P}^1$}-bundles admitting another smooth morphism of relative dimension one}, {J. Algebra}, {\bf 414}, (2014), no. 15, 105--119.

\bibitem[W2]{Wa2} Watanabe, K., {\it Fano $5$-folds with nef tangent bundles and {P}icard numbers greater than one}, {Math. Z.}, {\bf 276}, (2014), no. 1-2, 39--49.

\end{thebibliography}
\end{document}